\documentclass[12pt,a4paper]{article}
\usepackage[dvips]{graphicx,color} 
\usepackage[T1]{fontenc}  
\usepackage{times}    
\usepackage{enumerate,amsmath,amsthm,mathrsfs,citesort}
\usepackage{amsfonts}
\usepackage{amssymb}
\usepackage[a4paper,left=2.5cm,right=2cm,top=2.5cm,bottom=2.5cm]{geometry}
\usepackage{picins}
\usepackage{amstext}
\usepackage{array}
\usepackage[disable]{todonotes} 
\usepackage[symbol]{footmisc}
\newcounter{count}
\newtheoremstyle{bthm}{\baselineskip}{\baselineskip}{\slshape}{}{\bfseries}{}{ }{}
\newtheoremstyle{bex}{\baselineskip}{\baselineskip}{}{}{\sffamily}{:}{\newline }{}
\theoremstyle{bthm}
\newtheorem{theorem}{Theorem}[section]
\newtheorem{corollary}[theorem]{Corollary}
\newtheorem{lemma}[theorem]{Lemma}
\newtheorem{proposition}[theorem]{Proposition}

\newtheorem{fact}[theorem]{Fact}

\theoremstyle{bex}

\begin{document}
\begin{titlepage}
\title{$\chi$-binding function for a superclass of $2K_2$-free graphs.}
\author{Athmakoori Prashant$^{1,}$\footnote{The author's research was supported by the Council of Scientific and Industrial Research,  Government of India, File No: 09/559(0133)/2019-EMR-I.} and S. Francis Raj$^{2}$}
\date{{\footnotesize Department of Mathematics, Pondicherry University, Puducherry-605014, India.}\\
{\footnotesize$^{1}$: 11994prashant@gmail.com\, $^{2}$: francisraj\_s@pondiuni.ac.in\ }}
\maketitle
\renewcommand{\baselinestretch}{1.3}\normalsize
\begin{abstract}
The class of $2K_2$-free graphs has been well studied in various contexts in the past.
In this paper, we study the chromatic number of $\{butterfly, hammer\}$-free graphs, a superclass of $2K_2$-free graphs and show that a connected $\{butterfly, hammer\}$-free graph $G$ with $\omega(G)\neq 2$ admits $\binom{\omega+1}{2}$ as a $\chi$-binding function which is also the best available $\chi$-binding function for its subclass of $2K_2$-free graphs.
In addition, we show that if $H\in\{C_4+K_p, P_4+K_p\}$, then any $\{butterfly, hammer, H\}$-free graph $G$ with no components of clique size two admits a linear $\chi$-binding function. Furthermore, we also establish that any connected $\{butterfly, hammer, H\}$-free graph $G$ where $H\in \{(K_1\cup K_2)+K_p, 2K_1+K_p\}$, is perfect for $\omega(G)\geq 2p$. 
\end{abstract}
\section{Introduction}\label{intro}
All graphs considered in this paper are simple, finite and undirected.
Let $G$ be a graph with vertex set $V(G)$ and edge set $E(G)$.
For any positive integer $k$, a \emph{proper $k$-coloring} or simply \emph{$k$-coloring} of a graph $G$ is a mapping $c$ : $V(G)\rightarrow\{1,2,\ldots,k\}$ such that for any two adjacent vertices $u,v\in V(G)$, $c(u)\neq c(v)$.
If a graph $G$ admits a proper $k$-coloring, then $G$ is said to be \emph{$k$-colorable}.
The \emph{chromatic number}, $\chi(G)$, of a graph $G$ is the smallest $k$ such that $G$ is $k$-colorable.
In this paper, $K_n,P_n$ and $ C_n$ denote the complete graph, the path and the cycle on $n$ vertices respectively.
For $T,S\subseteq V(G)$, let $N_S(T)=N(T)\cap S$ (where $N(T)$ denotes the set of all neighbors of $T$ in $G$), let $\langle T\rangle$ denote the subgraph induced by $T$ in $G$ and let $[T,S]$ denote the set of all edges with one end in $T$ and the other end in $S$.
If every vertex in $T$ is adjacent with every vertex in $S$, then $ [T, S ]$ is said to be \emph{complete}.
For any graph $G$, we write  $H\sqsubseteq G$ if $H$ is an induced subgraph of $G$.
A set $S\subseteq V(G)$ is said to be an \emph{independent set} (\emph{clique}) if any two vertices of $S$ are non-adjacent (adjacent).
The \emph{clique number} of a graph $G$, is the size of a maximum clique in $G$ and is denoted by $\omega(G)$.
When there is no ambiguity, $\omega(G)$ will be denoted by $\omega$.
A subset $S$ of $V(G)$ is known as a dominating set if every vertex in $V\backslash S$ has a neighbor in $S$.

Let $\mathcal{F}$ be a family of graphs.
We say that $G$ is \emph{$\mathcal{F}$}-free if it contains no induced subgraph which is isomorphic to a graph in $\mathcal{F}$.
For a fixed graph $H$, let us denote the family of $H$-free graphs by $\mathcal{G}(H)$.
For two vertex-disjoint graphs $G_1$ and $G_2$, the \emph{join} of $G_1$ and $G_2$, denoted by $G_1+G_2$, is the graph whose vertex set $V(G_1+G_2) = V(G_1)\cup V(G_2)$ and the edge set $E(G_1+G_2) = E(G_1)\cup E(G_2)\cup\{xy: x\in V(G_1),\ y\in V(G_2)\}$.

A graph $G$ is said to be \emph{perfect} if $\chi(H)=\omega(H)$, for every induced subgraph $H$ of $G$.
A graph $G$ is said to be a \emph{multisplit graph}, if $V(G)$ can be partitioned into two subsets $V_1$ and $V_2$ such that $V_1$ induces a complete multipartite graph and $V_2$ is an independent set in $G$.
In particular, if $V_1$ induces a complete graph, then $G$ is said to be a \emph{split graph}.
A hereditary graph class $\mathcal{G}$ is said to be \emph{$\chi$-bounded} \cite{gyarfas1987problems} if there is a function $f$ (called a $\chi$-binding function) such that $\chi(G)\leq f(\omega(G))$, for every $G\in \mathcal{G}$.
We say that the $\chi$-binding function $f$ is \emph{special linear} if $f(x)= x+c$, where $c$ is a constant.
There has been extensive research done on $\chi$-binding functions for various graph classes.
See for instance, \cite{schiermeyer2019polynomial,randerath2004vertex,gyarfas1987problems,chung1990maximum,el1985existence,wagon1980bound,gaspers20192p_2,erdos1985problems}.
Let us recall a famous result by P.Erd\H{o}s
\begin{theorem}[\cite{erdos1959graph}]\label{erdos} 
For any positive integers $k,l\geq 3$, there exists a graph $G$ with girth at least $l$ and $\chi(G)\geq k$ (girth of $G$ is the length of the shortest cycle in $G$).
\end{theorem}
As a consequence of Theorem \ref{erdos}, Gy\'arf\'as in \cite{gyarfas1987problems} observed that there exists no $\chi$-binding function for $\mathcal{G}(H)$ whenever $H$ contains a cycle.
He further went on to conjecture that $\mathcal{G}(H)$ is $\chi$-bounded for every fixed forest $H$.

One of the earlier works in this direction, was done by S.Wagon in \cite{wagon1980bound} where he showed that the class of $2K_2$-free graphs admits $\binom{\omega+1}{2}$ as a $\chi$-binding function.
He further extended this result to show that the class of $pK_2$-free graphs admit a $\chi$-binding function of $O(\omega^{2p-2})$, when $p\in \mathbb{N}$.  
There has been extensive studies on the class of $2K_2$-free graphs ; see for instance \cite{chung1990maximum,el1985existence,wagon1980bound,gaspers20192p_2,erdos1985problems,blazsik1993graphs,karthick2018chromatic,brause2019chromatic,prashant2021chromatic,caldam2021chromaticspringerversion}.
In \cite{dhanalakshmi20162k2}, S. Dhanalakshmi et al. established the following structural characterization for $2K_2$-free graphs.
\begin{theorem}[\cite{dhanalakshmi20162k2}]\label{2k2characterization}
A connected graph is $2K_2$-free if and only if it is $\{P_5,butterfly,hammer\}$-free.
\end{theorem} 

From Theorem \ref{2k2characterization}, we see that $\{P_5,hammer\}$-free graphs, $\{P_5,butterfly\}$-free graphs and $\{butterfly, hammer\}$-free graphs are superclasses of $2K_2$-free graphs.
In \cite{brause2019chromatic} Brause et al. showed that the class of $\{P_5,hammer\}$-free graphs admits the same $\chi$-binding function as the class of $2K_2$-free graphs, namely $\binom{\omega+1}{2}$.
I. Schiermeyer in \cite{schiermeyer2017chromatic} established a cubic $\chi$-binding function for $\{P_5,butterfly\}$-free graphs which was recently improved to $\frac{3}{2}(\omega^2-\omega)$ by Wei Dong et al. in \cite{dong2022chromatic}.
Finally, it remains to study the $\chi$-binding function for $\{butterfly, hammer\}$-free graphs.
As observed by Gy\'arf\'as in \cite{gyarfas1987problems}, the class of $\{butterfly, hammer\}$-free graphs does not admit any $\chi$-binding function.
Thus the natural question would be to characterize $\{butterfly, hammer\}$-free graphs which are $\chi$-bounded.

In this direction we begin this paper by showing that if $G$ is a connected $\{butterfly, hammer\}$-free graph with $\omega(G)\neq 2$, then $G$ is $\chi$-bounded and the $\chi$-binding function is same as the class of $2K_2$-free graphs.
As a result the connected $\{butterfly, hammer\}$-free graphs that are $\chi$-unbounded should be of clique size $2$.
In addition, we were able to show that if $G$ is a connected $\{butterfly, hammer\}$-free graph $G$ with $\omega(G)\neq 2$ then $V(G)$ can be partitioned into four sets where three of them are independent and one induce a $2K_2$-free graph.
As a consequence we see that for any graph $H$, if the class of $\{2K_2,H\}$-free graphs admits $f(x)$ as a $\chi$-binding function then the class of connected $\{butterfly, hammer, H\}$-free graph $G$ with $\omega(G)\neq 2$ will admit $f(x)+3$ as a $\chi$-binding function. 
In \cite{karthick2018chromatic} T. Karthick and S. Mishra posed a problem seeking tight$\backslash$linear chromatic bounds for $\{2K_2,H\}$-free graphs for various $H$, where $H$ is a graph on $t\geq 6$ vertices.
We partially answer this question by providing tight linear $\chi$-bounds for $\{butterfly,hammer,C_4+K_p\}$-free graphs and $\{butterfly, hammer, P_4+K_p\}$-free graphs $G$ which contains no components of clique size two.
Furthermore, when $\omega(G)\geq 2p$, we show that a connected $\{butterfly,hammer,(K_1\cup K_2)+K_p\}$-free graph $G$ is a multisplit graph and a connected $\{butterfly,hammer,2K_1+K_p\}$-free graph $G$ is a  split graph.
Finally, we also show that for $H\in \{(K_1\cup K_2)+K_p, 2K_1+K_p\}$, a connected $\{butterfly, hammer, H\}$-free graph $G$ with $\omega(G)\neq 2$ admits a linear $\chi$-binding function.

Some graphs that are considered as forbidden induced subgraphs in this paper are given in Figure \ref{somesplgraphs}.
\begin{figure}[t]
	\centering
		\includegraphics[width=0.8\textwidth]{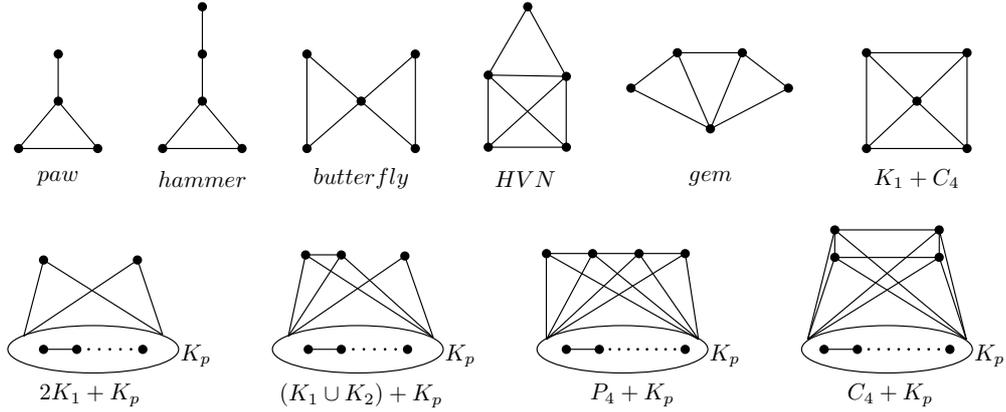}
	\caption{Some special graphs}
	\label{somesplgraphs}
\end{figure}
Notations and terminologies not mentioned here are as in \cite{west2005introduction}.

\section{Preliminaries}

Throughout this paper, we use a particular partition of the vertex set of a graph $G$ which was initially defined by S. Wagon in \cite{wagon1980bound} and later improved by A. P. Bharathi and S. A. Choudum  in \cite{bharathi2018colouring} as follows.
Let $A=\{v_1,v_2,\ldots,v_{\omega}\}$ be a maximum clique of $G$.
Let us define the \emph{lexicographic ordering} on the set $L=\{(i, j): 1 \leq i < j \leq \omega\}$ in the following way.
For two distinct elements $(i_1,j_1),(i_2,j_2)\in L$, we say that $(i_1,j_1)$ precedes $(i_2,j_2)$, denoted by $(i_1,j_1)<_L(i_2,j_2)$ if either $i_1<i_2$ or $i_1=i_2$ and $j_1<j_2$.
For every $(i,j)\in L$, let $C_{i,j}=\left\{v\in V(G)\backslash A:v\notin N(v_i)\cup N(v_j)\right\}\backslash\left\{\mathop\cup\limits_{(i',j')<_L(i,j)} C_{i',j'}\right\}$.

For $1\leq k\leq \omega$, let us define $I_k=\{v\in V(G)\backslash A: v\in N(v_i), \text{\ for\ every}\ i\in\{1,2,\ldots,\omega\}\backslash\{k\}\} $.
Since $A$ is a maximum clique, for $1\leq k\leq \omega$, $I_k$ is an independent set and for any $x\in I_k$, $xv_k\notin E(G)$.
Clearly, each vertex in $V(G)\backslash A$ is non-adjacent to at least one vertex in $A$. 
Hence those vertices will be contained either in $I_k$ for some $k\in\{1,2,\ldots,\omega\}$, or in $C_{i,j}$ for some $(i,j)\in L$.
Thus $V(G)=A\cup\left(\mathop\cup\limits_{k=1}^{\omega}I_k\right)\cup\left(\mathop\cup\limits_{(i,j)\in L}C_{i,j}\right)$.
Sometimes, we use the partition  $V(G)=V_1\cup V_2$, where $V_1=\mathop\cup\limits_{1\leq k\leq\omega}(\{v_k\}\cup I_k)=\mathop\cup\limits_{1\leq k\leq\omega}U_k$ and $V_2=\mathop\cup\limits_{(i,j)\in L}C_{i,j}$.

Without much difficulty, one can make the following observations.
\begin{fact}\label{true} For every $(i,j)\in L$, if a vertex $a\in C_{i,j}$, then $N_A(a)\supseteq \{v_1,v_2,\ldots,v_j\}\backslash\{v_i,v_j\}$.
\end{fact}
\section{$\{butterfly,hammer\}$-free graphs}

In \cite{chung1990maximum}, Chung et al. proved the existence of a dominating maximum clique in a connected $2K_2$-free graph $G$, when $\omega(G)\geq 3$.
We begin Section 3 by showing that if $G$ is a connected graph in $\mathcal{G}(hammer)$, a superclass of $2K_2$-free graphs, with $\omega(G)\neq 2$, then it contains a dominating $\omega$-partite graph with clique size $\omega$.   
\begin{theorem}\label{hammerfree}
If $G$ is a connected $hammer$-free graph with $\omega(G)\neq 2$, then it contains a dominating $\omega$-partite graph with clique size $\omega$.
\end{theorem}
\begin{proof}
Let $G$ be a connected $hammer$-free graph and $V(G)=V_1\cup V_2$, where $V_1=\mathop\cup\limits_{1\leq k\leq\omega}(\{v_k\}\cup I_k)=\mathop\cup\limits_{1\leq k\leq\omega}U_k$ and $V_2=\mathop\cup\limits_{(i,j)\in L}C_{i,j}$.
Clearly, $\langle V_1\rangle$ is an $\omega$-partite graph such that $\omega(\langle V_1\rangle)=\omega(G)$.
We shall show that $V_1$ is a dominating set of $G$.
To show this it is enough to show that every vertex in $V_2$ has a neighbor in $V_1$.
As $G$  is a connected graph, for $\omega(G)=1$, there is nothing to prove.
So let us assume that $\omega(G)\geq 3$.
Suppose there exists $x\in V_2$ such that $[x,V_1]=\emptyset$, the fact that $G$ is connected guarantees that there exists $y\in V_2$ such that $[y,V_1]\neq \emptyset$ and $x$ and $y$ are connected by a path.
Let $y'\in V_2$ such that $[y',V_1]\neq \emptyset$ and $[y,V_1]=0$ for every $y\in V_2$ such that $d(x,y)\leq d(x,y')$.
Let $P_{xy'}$ denote the shortest path between $x$ and $y'$ and $z$ be the preceeding vertex of $y'$ in $P_{xy}$.
Let us assume that $y'\in C_{m,n}$.
If $N_A(y')\neq \emptyset$, say $v_s\in N_A(y')$, then $\langle\{z,y',v_s,v_m,v_n\}\rangle\cong hammer$, a contradiction.
We know that $[y', V_1]\neq \emptyset$.
Therefore  there exists a vertex $a\in I_s$, where $s\in \{1,2,\ldots,\omega\}$ such that $ay'\in E(G)$.
Also since $\omega(G)\geq 3$, there exists two integers $r,q\in \{1,2,\ldots,\omega\}\backslash\{s\}$ such that $\langle\{z,y',a,v_r,v_q\}\rangle\cong hammer$, a contradiction.
Hence every vertex in $V_2$ has a neighbor in $V_1$.
\end{proof}

One can notice that $\omega(hammer)=\omega(butterfly)=3$.
Hence considering all connected graphs $G$ which are $\{butterfly, hammer\}$-free graph with clique size $2$, is equivalent to considering the the class of all $K_3$-free graphs which is clearly $\chi$-unbounded. The Mycielskian construction is a classical way to construct triangle free graphs with chromatic number $k$, where $k\in N$.
Hence, while considering the $\chi$-binding functions of $\{butterfly,hammer\}$- free graphs we shall assume that $\omega\neq 2$.
Let us now find a structural characterization and a $\chi$-binding function for $\{butterfly, hammer\}$-free graphs.

\begin{theorem}\label{hbf}
If $G$ is a connected $\{butterfly,hammer\}$-free graph with $\omega(G)\neq 2$, then
\begin{enumerate}[(i)] 
\item There exists a partition of $V(G)$ such that $V(G)= X_1\cup X_2\cup X_3\cup X_4\cup X_5$ where $\langle X_1\rangle$ is a dominating $\omega$-partite graph with $\omega(\langle X_1\rangle)=\omega(G)$, $\langle X_2\rangle$, $\langle X_3\rangle$ and $\langle X_4\rangle$ are independent sets and $(\langle X_1\cup X_5\rangle)$ is $2K_2$-free
\item $\chi(G)\leq \binom{\omega(G)+1}{2}$.
\end{enumerate}
\end{theorem}
\begin{proof}
Let $G$ be a connected $\{butterfly,hammer\}$-free graph with $\omega(G)\neq 2$ and let $V(G)=V_1\cup V_2$, where $V_1=\mathop\cup\limits_{1\leq k\leq\omega}(\{v_k\}\cup I_k)=\mathop\cup\limits_{1\leq k\leq\omega}U_k$ and $V_2=\mathop\cup\limits_{(i,j)\in L}C_{i,j}$.
For $\omega(G)=1$, there is nothing to prove.
So let us assume that $\omega(G)\geq 3$.
As in Theorem \ref{hammerfree}, one can observe that $V_1$ dominates $G$.
We shall now show that $C_{i,j}$'s are independent.

\textbf{Claim 1}: $C_{i,j}$ is an independent set for every $(i,j)\in L$.

By Fact \ref{true} for each $C_{i,j}$, where $(i,j)\neq (1,2)$, there exists an integer $s\in \{1,2,\ldots,j\}\backslash\{i,j\}$ such that $[v_s,C_{i,j}]$ is complete.
Thus if there exists an edge $ab$ in $\langle C_{i,j}\rangle$, $(i,j)\neq (1,2)$, then $\langle\{a,b,v_s,v_i,v_j\}\rangle\cong butterfly$, a contradiction
Next, let us show that $C_{1,2}$ is also an independent set.
On the contrary, let $ab$ be an edge in $\langle C_{1,2}\rangle$.
Since $V_1$ dominates $G$, $[a,V_1]\neq \emptyset$ and $[b,V_1]\neq \emptyset$.
Now if $[\{a,b\}, A]\neq \emptyset$, say $v_s$ is adjacent to either $a$ or $b$ then $\langle\{a,b,v_s,v_1,v_2\}\rangle$ will induce a $butterfly$ or a $hammer$.
Hence $[\{a,b\}, I_s]\neq \emptyset$ for some $s\in\{1,2,\ldots,\omega\}$ say $z\in I_s$ such that $[\{a,b\}, I_s]\neq \emptyset$.
Since $\omega(G)\geq 3$, we can find $r,q\in \{1,2,\ldots,\omega\}\backslash \{s\}$ such that $\langle\{a,b,z,v_r,v_q\}\rangle\cong butterfly$ or $hammer$, a contradiction.
Hence $C_{1,2}$ is also independent.
 
Next, we shall show that $\langle V_1\cup (\mathop\cup\limits_{j\geq 4} \mathop\cup\limits_{i\leq j-1} C_{i,j})\rangle$ is $2K_2$-free.

On the contrary, let us assume that there exists an induced  $2K_2$ in $\langle V_1\cup (\mathop\cup\limits_{j\geq 4} \mathop\cup\limits_{i\leq j-1} C_{i,j})\rangle$, say $a,b,c,d$ with edges $ab,cd\in E(G)$.
We begin by considering $\omega(G)=3$.
Here, $V_1\cup (\mathop\cup\limits_{j\geq 4} \mathop\cup\limits_{i\leq j-1} C_{i,j})=V_1$.
One can observe that $\{a,b,c,d\}\cap A=\emptyset$.
Hence $a,b,c,d\in I_1\cup I_2\cup I_3$.
Clearly, there exists an integer $k\in \{1,2,3\}$ such that it contains exactly two non-adjacent vertices of $\{a,b,c,d\}$.
Depending on whether the remaining two vertices of $\{a,b,c,d\}$ belong to the same $I_i$ or not, we see that there exists an $r\in\{1,2,3\}$ such that $\langle\{v_r,a,b,c,d\}\rangle\cong butterfly$ or $hammer$, a contradiction.
Hence, $\langle V_1\rangle$ is $2K_2$-free and $V_1\cup C_{1,2}\cup C_{1,3}\cup C_{2,3}$ is the required partition.

Next, let us consider $\omega(G)\geq 4$. Let us break our proof into three cases.

\textbf{Case 1}: $|\{a,b,c,d\} \cap (\mathop\cup\limits_{j\geq 4} \mathop\cup\limits_{i\leq j-1} C_{i,j})|\leq 1$.

Without loss of generality, let us assume that $b,c,d\in V_1$.
Since $\omega(G)\geq 4$, there exists $q\in \{1,2,\ldots,\omega\}$ such that $U_q\cap \{b,c,d\}=\emptyset$.
Hence $\langle\{a,b,c,d,v_{q}\}\rangle\cong butterfly$ or $hammer$, a contradiction.

\textbf{Case 2}: $|\{a,b,c,d\} \cap (\mathop\cup\limits_{j\geq 4} \mathop\cup\limits_{i\leq j-1} C_{i,j})|=2$.

Let $x,y\in\{a,b,c,d\}$ such that $x,y\in V_1$.
Let $r_1,r_2\in\{1,2,\ldots,\omega\}$ such that $x,y\in U_{r_1}\cup U_{r_2}$.
Clearly, there exists an integer $q_1\in \{1,2,3\}\backslash\{r_1,r_2\}$ such that $[v_{q_1},\{x,y\}]$ is complete.
If $|[v_{q_1},\{a,b,c,d\}\backslash\{x,y\}]|\neq 0$, then $\langle\{v_{q_1},a,b,c,d\}\rangle\cong butterfly$ or $hammer$ , a contradiction.
Hence $|[v_{q_1},\{a,b,c,d\}\backslash\{x,y\}]|=0$, which implies $\{\{a,b,c,d\}\backslash\{x,y\}\}\subseteq \mathop\cup\limits_{j\geq 4}C_{{q_1},j}$.
If there exists an integer $q_2\in \{1,2,3\}\backslash\{r_1,r_2,q_1\}$, then $\langle\{v_{q_2},a,b,c,d\}\rangle\cong butterfly$, a contradiction.
Hence $\{r_1,r_2,q_1\}=\{1,2,3\}$.
Without loss of generality, let us assume that $\{x,y\}\cap U_{r_1}\neq \emptyset$, then $\langle\{v_{r_2},a,b,c,d\}\rangle\cong butterfly$ or $hammer$, a contradiction.

\textbf{Case 3}: $|\{a,b,c,d\}\cap (\mathop\cup\limits_{j\geq 4} \mathop\cup\limits_{i\leq j-1} C_{i,j})|\geq 3$.

Without loss of generality, let us assume that $a,b,c\in (\mathop\cup\limits_{j\geq 4} \mathop\cup\limits_{i\leq j-1} C_{i,j})$.
If $x\in \{v_1,v_2,v_3\}\cap N_A(a)\cap N_A(b)\cap N_A(c)$, then $\langle\{x,a,b,c,d\}\rangle\cong butterfly$ or $hammer$, a contradiction.
Hence $\{v_1,v_2,v_3\}\cap N_A(a)\cap N_A(b)\cap N_A(c)=\emptyset$.
Thus $a,b,c,\in (\mathop\cup\limits_{j\geq 4}C_{1,j})\cup(\mathop\cup\limits_{j\geq 4}C_{2,j})\cup(\mathop\cup\limits_{j\geq 4}C_{3,j})$ such that $\{a,b,c\}\cap (\mathop\cup\limits_{j\geq 4}C_{i,j})\neq \emptyset$, for every $i\in\{1,2,3\}$.
Clearly, $|[x,\{a,b,c\}]|=2$, for every $x\in\{v_1,v_2,v_3\}$.
Now we know that $d\in V_1\cup(\mathop\cup\limits_{j\geq 4} \mathop\cup\limits_{i\leq j-1} C_{i,j})$, which implies that there exists a vertex $y\in \{v_1,v_2,v_3\}$ such that $dy\in E(G)$.
Therefore, $\langle\{y,a,b,c,d\}\rangle\cong hammer$, a contradiction.

Thus $V(G)=V_1\cup C_{1,2}\cup C_{1,3}\cup C_{2,3}\cup (\mathop\cup\limits_{j\geq 4} \mathop\cup\limits_{i\leq j-1} C_{i,j})$ is the required partition where $X_1=V_1$, $X_2=C_{1,2}$, $X_3=C_{1,3}$, $X_4=C_{2,3}$, $X_5=(\mathop\cup\limits_{j\geq 4} \mathop\cup\limits_{i\leq j-1} C_{i,j})$ and $\langle X_1\cup X_5\rangle$ is $2K_2$-free.

Finally let us establish an $\binom{\omega+1}{2}$ coloring of $G$.
Let $\{1,2,\ldots,\binom{\omega+1}{2}\}$ be the set of colors.
We shall start by assigning the color $k$ to the vertices in $v_k\cup I_k$, where $k\in \{1,2,\ldots,\omega\}$.
Next, for $(i,j)\in L$, color each $C_{i,j}$ with a new color and hence we can color the vertices in $G$ with at most $\omega+\sum\limits_{j=2}^{\omega(G)} (j-1) = \binom{\omega(G)+1}{2}$ colors.
Clearly, this is a proper coloring of $G$.
Hence $\chi(G)\leq \binom{\omega+1}{2}$.
\end{proof}

One can observe that if $G$ is a $2K_2$-free graph then it contains at most one edge containing component and hence any $k$-coloring for the non-trivial component of $G$ will yield a $k$-coloring for $G$.
Now as a consequence of Theorem \ref{hbf} we get Corollary \ref{chibindrelation} and Corollary \ref{2k2}, the latter a result due to S.Wagon in \cite{wagon1980bound}.

\begin{corollary}\label{chibindrelation}
If the class of $\{2K_2,H\}$-free graphs admits the $\chi$-binding function $f(x)$, then a connected $\{butterfly, hammer, H\}$-free graph $G$ such that $\omega(G)\neq 2$ would admit $f(x)+3$ as a $\chi$-binding function. 
\end{corollary}

\begin{corollary}[\cite{wagon1980bound}]\label{2k2}
If $G$ is a $2K_2$-free graph, then $\chi(G) \leq \binom{\omega(G)+1}{2}$.
\end{corollary}
\begin{proof}
Without loss of generality let $G$ be a connected $2K_2$-free graph.
When $\omega(G)\neq 2$, the proof follows from Theorem \ref{hbf}.
When $\omega(G)=2$, $V(G)=C_{1,2}\cup A\cup I_1\cup I_2$, where $C_{1,2}$ is an independent set.
Now by coloring $v_i\cup I_i$ with color $i$ for $i=1,2$ and by coloring $C_{1,2}$ with $3$, we get a proper $3$-coloring for $G$.
Hence $\chi(G)\leq 3$.
\end{proof}

\subsection{$\{butterfly, hammer, C_4+K_p\}$-free graphs}\label{sect3.1}
In Section \ref{sect3.1} we establish a $\chi$-binding function for $\{butterfly, hammer, C_4+K_p\}$-free graphs.
\begin{theorem}\label{bhc4kpf}
Let $p$ be a positive integer and $G$ be a $\{butterfly,hammer,C_4+K_p\}$-free graph such that no component has clique size two, then $\chi(G)\leq \omega(G)+ \frac{p(p+1)}{2}$.			
\end{theorem}
\begin{proof}
Let $G$ be a connected $\{butterfly,hammer, C_4+K_p\}$-free graph such that $\omega(G)\neq 2$.
Let $V(G)=V_1\cup V_2$ where $V_1=A\cup\left(\mathop\cup\limits_{k=1}^{\omega}I_k\right)$ and $V_2=\left(\mathop\cup\limits_{(i,j)\in L}C_{i,j}\right)$.
Let  $\{1,2,\ldots,\omega(G)+\frac{p(p+1)}{2}\}$ be the set of colors.
For $1\leq k\leq \omega$, assign the color $k$ to the vertex $v_k$.
Since $G$ is a connected $\{butterfly,hammer\}$-free graph with $\omega(G)\neq 2$, by Claim 1 of Theorem \ref{hbf}, $C_{i,j}$ is an independent set for every $(i,j)\in L$.
In order to color the vertices in $V_2\cup \left(\mathop\cup\limits_{k=1}^{\omega}I_k\right)$, let us break the proof into two cases depending upon the value of $\omega(G)$.

\noindent\textbf{Case 1} $1\leq \omega(G)\leq p+1$

For $\omega(G)=1$, there is nothing to prove.
Let $\omega(G)\geq 3$.
For $1\leq k\leq \omega$, assign the color $k$ to the vertices in $I_k$.
Since each $C_{i,j}$ is an independent set, each $C_{i,j}$ can be properly colored by assigning a new color.
Therefore $V_2$ can be colored with at most $\sum\limits_{j=2}^{p+1} j-1= \frac{p(p+1)}{2}$ colors.   
Hence $\chi(G)\leq \omega(G)+\frac{p(p+1)}{2}$.

\noindent\textbf{Case 2} $\omega(G)\geq p+2$

We begin by showing that for $1\leq k,\ell\leq \omega$,  $[I_k, I_{\ell}]=\emptyset$.
Suppose for some $k,\ell\in\{1,2,\ldots,\omega\}$, there exist vertices $a\in I_k$ and $b\in I_{\ell}$ such that $ab\in E(G)$,
then there exists integers ${q_1},{q_2},\ldots,{q_p}\in \{1,2,\ldots,\omega\}\backslash\{k,\ell\}$ such that $\langle\{a,v_{\ell},v_k,b,v_{q_1},v_{q_2},\ldots,v_{q_p}\}\rangle\cong C_4+K_p$, a contradiction.
Hence $(\mathop\cup\limits_{k=1}^\omega I_k)$ is an independent set and can be colored with a single color. Let it be  $\omega+1$.
Next, for $j\geq p+3$, we shall show that $[C_{i,j}, C_{m,j}]=\emptyset$, where $i,m\in\{1,2,\ldots,j-1\}$ and $i\neq m$.
Suppose there exist vertices $a\in C_{i,j}$ and $b\in C_{m,j}$ such that $ab\in E(G)$.
By Fact \ref{true} we can find integers ${q_1},{q_2},\ldots,{q_p}\in \{1,2,\ldots,j\}\backslash\{i,m,j\}$ such that $\langle \{a,v_m,v_i,b,v_{q_1},v_{q_2},\ldots,v_{q_p}\}\rangle\cong C_4+K_p$, a contradiction.
Thus, for every $j\geq p+3 $, $(\mathop\cup\limits_{i=1}^{j-1} C_{i,j})$ is an independent set.
One can observe that, $[v_j,(\mathop\cup\limits_{i=1}^{j-1} C_{i,j})]=\emptyset$ and hence for any $j\geq p+3$, we can assign the color $j$ to the vertices in $(\mathop\cup\limits_{i=1}^{j-1}C_{i,j})$.
Next, by definition of $C_{i,j}$, we get that $[\{v_i,v_j\},C_{i,j}]=\emptyset$.
Hence we can assign the color $1$, $p+2$ and $i$ for the vertices in $C_{1,p+1}$, $C_{1,p+2}$ and $C_{i,p+2}$ ($i\geq 2$) respectively.
Now for the remaining $C_{i,j}$, we can color them by assigning a new color to each of the $C_{i,j}$ and hence $G$ can be colored with at most $\omega+1+\sum\limits_{j=2}^{p+1} (j-1) -1= \omega+\frac{p(p+1)}{2}$.
Next, if $G$ is a disconnected $\{butterfly, hammer, C_4+K_p\}$-free graph such that no component has clique size two, then by similar arguments we can show that each component is $\left(\omega(G)+\frac{p(p+1)}{2} \right)$- colorable and hence $\chi(G)\leq \omega(G)+\frac{p(p+1)}{2}$.
\end{proof}
As a simple consequence of Corollary \ref{2k2} and Theorem \ref{bhc4kpf} we get the following results.
\begin{theorem}\label{2k2kpc4}
Let $p$ be a positive integer and $G$ be a $\{2K_2,C_4+K_p\}$-free graph, then $\chi(G)\leq \omega(G)+ \frac{p(p+1)}{2}$.
\end{theorem}
\begin{corollary}\cite{prashant2021chromatic,caldam2021chromaticspringerversion}\label{2k2k1c4} 
If $G$  is a $\{2K_2,K_1+C_4\}$-free graph with, then $\chi(G)\leq \omega(G)+1$.
\end{corollary}

Note that in \cite{prashant2021chromatic}, it has been shown that the bound obtained in Corollary \ref{2k2k1c4} is tight. 
Hence the $\chi$-binding function given in Theorem \ref{bhc4kpf} and Theorem \ref{2k2kpc4} could not be improved.
However, the question of whether this bound is optimal remains open for $p\geq 2$ .

\subsection{$\{butterfly, hammer, P_4+K_p\}$-free graphs}\label{sect3.2}
Let us begin Section \ref{sect3.2} by making some simple observations on $(P_4+K_p)$-free graphs.
\begin{lemma}\label{p4kplemma}
Let $p$ be a positive integer and $G$ be a $(P_4+K_p)$-free graph with $\omega(G)\geq p+2$.
Then the following hold

\begin{enumerate}[(i)]

\item\label{relationIlIr} For $k,\ell\in\{1,2,\ldots,\omega(G)\}$, $[I_k,I_{\ell}]$ is complete. Thus $\langle V_1\rangle$ is a complete multipartite graph with  $U_{k}=\{v_{k}\}\cup I_{k}$, $1\leq k\leq \omega(G)$ as its partitions.

\item\label{GIp} For $(i,j)\in L$ with $j\geq p+2$, if $a\in C_{i,j}$ such that  $av_{\ell}\notin E(G)$, for some $\ell\in\{1,2,\ldots,\omega\}$, then $[a,I_{\ell}]=\emptyset$.
\end{enumerate}
\end{lemma}
\begin{proof}
Let $G$ be a $(P_4+K_p)$-free graph with $\omega(G)\geq p+2$, $p\geq 1$.
\begin{enumerate}[(i)]

\item Suppose $[I_k,I_{\ell}]$ is not complete for some $k,\ell\in\{1,2,\ldots,\omega\}$ and there exists a pair of non-adjacent vertices $a\in I_k$ and $b\in I_{\ell}$ , then as $\omega(G)\geq p+2$, there exists integers $q_1, q_2,\ldots, q_p$ in $\{1,2,\ldots, \omega\}\backslash\{k,\ell\}$ such that  $\langle\{a,v_{\ell},v_k,b,v_{q_1},\ldots, v_{q_p}\}\rangle\cong P_4+K_p$, a contradiction.

\item On the contrary assume that there exist vertices $a\in C_{i,j}$ and $b\in I_{\ell}$ such that $av_{\ell}\notin E(G)$ and $ab\in E(G)$.
By Fact \ref{true}, $\ell\in \{i,j,j+1,\ldots,\omega\}$.
Let $k\in\{i,j\}\backslash\{\ell\}$ and since $j\geq p+2$, $\langle\{a,b,v_k,v_{\ell},\{\{v_1,v_2,\ldots,v_{p+1}\}\backslash\{v_i\}\}\}\rangle$ contains an induced $P_4+K_p$, a contradiction.
\end{enumerate}
\end{proof}

Now, let us establish a linear $\chi$-binding function for $\{butterfly, hammer, P_4+K_p\}$-free graphs.
When $p=0$, $P_4+K_p\cong P_4$ and since $P_4$-free graphs are perfect, $G$ is $\omega$-colorable.
So let us consider $p\geq 1$.
\begin{theorem}\label{bhp4kpf}
Let $p$ be a positive integer and $G$ be a $\{butterfly,hammer,P_4+K_p\}$-free graph such that no component has clique size two, then\\
$\chi(G)\leq \left\{
\begin{array}{lcl}
\omega(G)+ \frac{p(p+1)}{2}					& \textnormal{for} & 1\leq\omega(G)\leq p+1	\\ 
\omega(G)+  \frac{p(p+1)}{2} -1  	& \textnormal{for} & \omega(G)\geq p+2.
\end{array}
\right.$
\end{theorem}
\begin{proof}
Let $G$ be a connected $\{butterfly,hammer,P_4+K_p\}$-free graph such that $\omega(G)\neq 2$.
Let $V(G)=V_1\cup V_2$, where $V_1=\mathop\cup\limits_{1\leq k\leq\omega}(\{v_k\}\cup I_k)=U_k$ and $V_2=\mathop\cup\limits_{(i,j)\in L}C_{i,j}$.
Clearly, the vertices of $V_1$ can be colored with $\omega(G)$ colors by assigning the color $k$ to the vertices in $U_k$, for $1\leq k\leq \omega$. 
Since $G$ is a connected $\{butterfly,hammer\}$-free graph with $\omega(G)\neq 2$, one can observe that by Claim 1 of Theorem \ref{hbf}, $C_{i,j}$ is an independent set for every $(i,j)\in L$.
In order to color the vertices of $V_2$, let us break the proof into two cases depending upon the value of $\omega(G)$.

\noindent\textbf{Case 1} $1\leq \omega(G)\leq p+1$

For $\omega(G)=1$, there is nothing to prove.
Since  $C_{i,j}$ is an independent set for every $(i,j)\in L$, the vertices in $C_{i,j}$ can be properly colored by assigning a new color each and therefore $V_2$ can be colored with atmost $\sum\limits_{j=2}^{p+1} j-1= \frac{p(p+1)}{2}$ colors.   
Hence $\chi(G)\leq \omega(G)+\frac{p(p+1)}{2}$.

\noindent\textbf{Case 2} $\omega(G)\geq p+2$

For $j\geq p+3$, we shall show that $\left(\mathop\cup\limits_{i=1}^{j-1} C_{i,j}\right)$ is an independent set.
Suppose for $j\geq p+3$ there exists two distinct $(i,j),(m,j)\in L$ such that  $a\in C_{i,j}$, $b\in C_{m,j}$ and $ab\in E(G)$, then we can find integers $q_1,q_2,\ldots,q_p\in \{1,2,\ldots,j-1\}\backslash \{i,m\}$ such that $\langle\{a,b,v_i,v_j,v_{q_1},v_{q_2},\ldots,v_{q_p}\}\rangle\cong P_4+K_p$, a contradiction.
Hence by using (\ref{GIp}) of Lemma \ref{p4kplemma}, for $j\geq p+3$ we can assign the color $j$ to each vertex in $\left(\mathop\cup\limits_{i=1}^{j-1} C_{i,j}\right)$.
Next, by using (\ref{GIp}) of Lemma \ref{p4kplemma}, we can assign the color $p+2$ to the vertices in $C_{1,p+2}$ and color $k$ to the vertices in $C_{k,p+2}$, $2\leq k\leq p+1$.
Let us next color the vertices of $C_{1,p+1}$. 
If $a\in C_{1,p+1}$ such that $[a,I_1]=\emptyset$, then assign the color $1$ to $a$.
If $[a,I_1]\neq \emptyset$ then we shall show that for $r\geq p+2$, $[a,U_r]=\emptyset$.
Let $a\in C_{1,2}$ and $b\in I_1$ such that $ab\in E(G)$.
If there exists an $r\geq p+2$ such that $[a,U_r]\neq \emptyset$, then let $u_r\in U_r$ such that $au_r\in E(G)$.
Then by using (\ref{relationIlIr}) of Lemma \ref{p4kplemma} and Fact \ref{true} we can see that if $p=1$, then $\langle\{a,b,v_{p+1},v_1,u_r\}\rangle\cong K_p+P_4$, a contradiction and if $p\geq 2$, then $\langle \{a,b,v_{p+1},v_1,u_r,v_2,\ldots,v_p\}\rangle\cong P_4+K_p$, a contradiction.
Hence $[a,U_r]=\emptyset$ for $r\geq p+2$ and this in turns implies that $[a,C_{p+1,p+2}]=\emptyset$ (Suppose there exists a vertex $c\in C_{p+1,p+2}$ such that $ac\in E(G)$, then $\langle{a,c,v_1,v_{p+1},v_{p+2}}\rangle\cong hammer$, a contradiction). 
Similarly, one can show that if $[a,I_1]\neq \emptyset$, then $[a,I_{p+1}]=\emptyset$.
Therefore we can assign the color $p+1$ to vertex $a$ and the vertices in $C_{1,p+1}$ can be colored with colors $1$ and $p+1$. 
Finally, the vertices of the remaining $C_{i,j}$ can be colored by assigning a new color to each $C_{i,j}$ and hence $G$ can be colored with at most $\omega+\sum\limits_{j=2}^{p+1} (j-1) -1= \omega+\frac{p(p+1)}{2}-1$ colors.
Clearly this is a proper coloring of $G$.
Since the clique size of each component is less than or equal to $\omega(G)=\omega$, we see that each component is $\left(\omega(G)+\frac{p(p+1)}{2}-1\right)$-colorable.
\end{proof}

By using Corollary \ref{2k2} and Theorem \ref{bhp4kpf} we get a $\chi$-binding function for $\{2K_2,P_4+K_p\}$-free graphs.
\begin{theorem}\label{2k2p4kp}
Let $p$ be a positive integer and $G$ be a $\{2K_2,P_4+K_p\}$-free graph, then\\
$\chi(G)\leq \left\{
\begin{array}{lcl}
\omega(G)+ \frac{p(p+1)}{2}					& \textnormal{for} & 1\leq\omega(G)\leq p+1	\\ 
\omega(G)+  \frac{p(p+1)}{2} -1  	& \textnormal{for} & \omega(G)\geq p+2.
\end{array}
\right.$
\end{theorem}

As a simple consequence of Theorem \ref{2k2p4kp} we get Corollary \ref{3omega} and Corollary \ref{2k2k2p4}, results in \cite{brause2019chromatic} and \cite{prashant2021chromatic} respectively.
\begin{corollary}\cite{brause2019chromatic}\label{3omega}
Let $G$ be a $\{2K_2, gem\}$-free graph, then $\chi(G)\leq \max\{3,\omega\}$.
\end{corollary}
\begin{corollary}\cite{prashant2021chromatic}\label{2k2k2p4}
If $G$ is a $\{2K_2, K_2+P_4\}$-free graph with $\omega(G)\geq 4$, then $\chi(G)\leq \omega(G)+2$. 
\end{corollary}

\subsection{$\{butterfly,hammer,(K_1\cup K_2)+K_p\}$-free graphs}\label{k1k2+kp}
In \cite{brause2019chromatic}, C. Brause et al. showed that $\{2K_2,(K_1\cup K_2)+K_p\}$-free graphs are multisplit graphs when $\omega(G)\geq 2p$, $p\geq 2$ and they admit $\omega+(2p-1)(p-1)$ as a $\chi$-binding function when $p\geq 2$.
Furthermore, they realized that this might not be the best possible bound and claimed that the problem of finding the optimal $\chi$-binding function is still open.
We shall show in subsection \ref{k1k2+kp} that $\{butterfly,hammer,(K_1\cup K_2)+K_p\}$-free graphs admits the same structure as $\{2K_2,(K_1\cup K_2)+K_p\}$-free graphs and get a better $\chi$-binding function than the existing $\chi$-binding function for $\{2K_2,H\}$-free graphs.

Let us begin by recalling some of the results in \cite{olariu1988paw,prashant2022chi,caldam2022chromaticspringerversion}.

\begin{theorem}\cite{olariu1988paw}\label{pawfree}
Let $G$ be a connected graph. Then $G$ is $paw$-free if and only if $G$ is either $K_3$-free or complete multipartite.
\end{theorem}

\begin{proposition}\cite{prashant2022chi,caldam2022chromaticspringerversion}\label{k1k2kpprop}
Let $G$ be a $((K_1\cup K_2)+K_p)$-free graph with $\omega(G)\geq p+2$, $p\geq 1$.
Then $G$ satisfies the following.
\begin{enumerate}[(i)]
\setlength\itemsep{-1pt}
\item \label{completeIkIl} For $k,\ell\in\{1,2,\ldots,\omega(G)\}$, $[I_k,I_{\ell}]$ is complete. Thus, $\langle V_1\rangle$ is a complete multipartite graph with  $U_{k}=\{v_{k}\}\cup I_{k}$, $1\leq k\leq \omega(G)$ as its partitions.

\item \label{emptycij} For $j\geq p+2$ and $1\leq i<j$, $C_{i,j}=\emptyset$.

\item \label{p-1neighbors} For $x\in V_2$, $x$ has neighbors in  at most $(p-1)$ $U_\ell$'s where $\ell\in\{1,2,\ldots,\omega(G)\}$.

\end{enumerate}
\end{proposition}

By using Proposition \ref{k1k2kpprop}, let us show that connected $\{butterfly, hammer, (K_1\cup K_2)+K_p\}$-free graphs with $\omega(G)\geq 2p$ are multisplit graphs.

\begin{theorem}\label{hbk1k2kp}
Let $p$ be an integer greater than $1$ and $G$ be a connected $\{butterfly,hammer,(K_1\cup K_2)+K_p\}$-free graph.
If $\omega(G)\geq 2p$, then $G$ is a multisplit graph.
\end{theorem}
\begin{proof}
Let $G$ be a connected $\{butterfly,hammer,(K_1\cup K_2)+K_p\}$-free graph with $\omega(G)\geq 2p$, $p\geq 2$.
Let $V(G)= V_1\cup V_2$.
Since $\omega(G)\geq 2p\geq p+2$, by (\ref{completeIkIl}) of Proposition \ref{k1k2kpprop}, we see that $\langle V_1\rangle$ is a complete multipartite graph with partitions $U_k=\{v_k\}\cup I_k$, $1\leq k\leq \omega(G)$.
Next, let us show that $V_2$ is an independent set.
Since $p\geq 2$, $\omega(G)\neq 2$ and hence by Claim 1 of Theorem \ref{hbf}, for $(i,j)\in L$, $C_{i,j}$ is an independent set.
Therefore if $a,b\in V_2$ such that $ab$ is an edge in $G$, then $a\in C_{m_1,n_1}$ and $b\in C_{m_2,n_2}$ where $(m_1,n_1)\neq (m_2,n_2)$.
It is easy to see that by (\ref{p-1neighbors}) of Proposition \ref{k1k2kpprop}, $|N_A(\{a,b\})|\leq 2p-2$ and hence there exist at least two vertices $v_r,v_s\in A$ such that $[\{a,b\},\{v_r,v_s\}]=\emptyset$.
Without loss of generality, let us assume that $(m_1,n_1)<_L(m_2,n_2)$. 
If $m_1<m_2$, then $bv_{m_1}\in E(G)$ and $\langle\{a,b,v_{m_1},v_r,v_s\}\rangle\cong hammer$, a contradiction.
If $m_1=m_2$ and $n_1<n_2$, then $bv_{n_1}\in E(G)$ and $\langle\{a,b,v_{n_1},v_r,v_s\}\rangle\cong hammer$, a contradiction.
Hence $V_2$ is independent and thereby $G$ is a multisplit graph.
\end{proof}
One can notice that when $p=0$ and $1$, $(K_1\cup K_2)+K_p\cong K_1\cup K_2$ and $(K_1\cup K_2)+K_p\cong paw$ respectively.
As $(K_1\cup K_2)\sqsubseteq P_4$, $(K_1\cup K_2)$-free graphs are perfect.
Also C. Brause et al. in \cite{brause2019chromatic} has shown that every multisplit graph is perfect. 
Hence as a consequence of Theorem \ref{pawfree} and Theorem \ref{hbk1k2kp} we get Theorem \ref{bhk1k2kpperfect}, Corollary \ref{2k2k1k2kpperfect} and Corollary \ref{correct}.
\begin{theorem}\label{bhk1k2kpperfect}
Let $p$ be a non-negative integer and $G$ be a connected $\{butterfly,hammer,(K_1\cup K_2)+K_p\}$-free graph with $\omega(G)\geq 2p$. If $\omega(G)\neq 2$, then $G$ is perfect.
\end{theorem}
\begin{corollary}\cite{brause2019chromatic}\label{2k2k1k2kpperfect}
If $G$ is a connected $\{2K_2,(K_1\cup K_2)+K_p\})$-free graph with $\omega(G)\geq 2p$, $p\geq 2$, then $G$ is a multisplit graph.
\end{corollary}
\begin{corollary}\cite{brause2019chromatic}\label{correct}
Let $p\geq 0$ be an integer and $G$ be a $\{2K_2,(K_1\cup K_2)+K_p\}$-free graph with $\omega(G)\geq 2p$. If $p\neq 1$ or $\omega(G)\neq 2$, then $G$ is perfect.
\end{corollary}
Finally, let us establish a $\chi$-binding function for the class of $\{butterfly,hammer,(K_1\cup K_2)+K_p\}$-free graphs. 
\begin{theorem}\label{colorbhk1k2}
Let $p$ be a non-negative integer and $G$ be a connected $\{butterfly,hammer,(K_1\cup K_2)+K_p\}$-free graph with $\omega(G)\neq 2$. Then\\
$\chi(G)\leq \left\{
\begin{array}{lcl}
\omega(G)+ \frac{p(p+1)}{2}					& \textnormal{for} & 1\leq\omega(G)\leq p+1	\\ 
\omega(G)+  \frac{p(p-1)}{2}  	& \textnormal{for} & \omega(G)\geq p+2.
\end{array}
\right.$
\end{theorem}
\begin{proof}
Let $G$ be a connected $\{butterfly,hammer, (K_1\cup K_2)+K_p\}$-free graph such that $\omega(G)\neq 2$. 
Let $\{1,2,\ldots,\omega(G)+\frac{p(p-1)}{2}\}$ be the set of colors.
Clearly the vertices in $V_1$ can be colored with $\omega$ colors by assigning color $k$ to the vertices of $U_k$ for $1\leq k\leq \omega$.
Since $(K_1\cup K_2)\sqsubseteq P_4$, for $1\leq \omega(G)\leq p+1$, the result follows from Theorem \ref{bhp4kpf}.
So we shall consider $\omega(G)\geq p+2$.
One can see that by (\ref{emptycij}) of Proposition \ref{k1k2kpprop}, $C_{i,j}=\emptyset$ for $j\geq p+2$
and by (\ref{p-1neighbors}) of Proposition \ref{k1k2kpprop}, $[C_{i,p+1}, U_{i}]=\emptyset$, for every $i\in\{1,2,\ldots p\}$.
Hence by Claim 1 of Theorem \ref{hbf}, one can assign color $i$ to the vertices of $C_{i,p+1}$, for every $i\in\{1,2,\ldots,p\}$.
The remaining $C_{i,j}$'s can be colored by assigning a new color to each $C_{i,j}$.
Clearly this is a proper coloring of $G$. 
Hence for $\omega(G)\geq p+2$, $G$ is $\left(\omega(G)+  \frac{p(p-1)}{2}\right)$-colorable. 
\end{proof}

As a consequence of Corollary \ref{2k2} and Theorem \ref{colorbhk1k2}, we get Corollary \ref{2k2k1k2kpcolor} and as a result we have a $\chi$-binding function better than $\omega(G)+(2p-1)(p-1)$, the one given by C. Brause et al. in \cite{brause2019chromatic} for the class of $\{2K_2, (K_1\cup K_2)+K_p\}$-free graphs, when $p\geq 2$.
\begin{corollary}\label{2k2k1k2kpcolor}
Let $p$ be a non-negative integer and $G$ be a $\{2K_2,(K_1\cup K_2)+K_p\}$-free graph,\\ then
$\chi(G)\leq \left\{
\begin{array}{lcl}
\omega(G)+ \frac{p(p+1)}{2}					& \textnormal{for} & 1\leq\omega(G)\leq p+1	\\ 
\omega(G)+  \frac{p(p-1)}{2} -1  	& \textnormal{for} & \omega(G)\geq p+2.
\end{array}
\right.$
\end{corollary}
\bibliographystyle{ams}
\bibliography{ref}
\end{titlepage}
\listoftodos
\end{document}